\documentclass[12pt]{article}
\usepackage[english]{babel}
\usepackage{amsmath, amssymb, amsthm, amsfonts}%
\usepackage[dvips]{graphicx}%
\usepackage{multicol}%
\newtheorem{theorem}{Theorem}

\theoremstyle{definition}
\newtheorem{defen}{Definition}
\newtheorem*{rem}{Remarks}

\renewcommand*{\th}{\mathop{\text{\textrm{th}}}}
\newcommand*{\cth}{\mathop{\text{\textrm{cth}}}}
\newcommand*{\sh}{\mathop{\text{\textrm{sh}}}}
\newcommand*{\ch}{\mathop{\text{\textrm{ch}}}}
\newcommand*{\ctg}{\mathop{\text{\textrm{ctg}}}}

\newcounter{glava}

\newcommand*{\glava}[1]{\par\vspace{20pt}\refstepcounter{glava}
\begin{center}\textsc{\theglava. #1}\end{center}\par
}

\begin{document}

\begin{center}
\large \textbf{HYPERBOLIC TRIANGLES OF THE MAXIMUM AREA \\ WITH TWO
FIXED SIDES}

\vspace{15pt}

\textsc{E.\,I.\,Alekseeva}
\end{center}

\vspace{5pt}

\begin{quote}
\textsc{Abstract}\footnote{This paper is prepared under the
supervision of P.\,Bibikov and is submitted to a prize of the Moscow
Mathematical Conference of High-school Students. Readers are invited
to send their remarks and reports to mmks@mccme.ru}. The aim of this
paper is to consider the Lobachevskii geometry analog of a
well-known Euclidian problem; namely: to find a triangle with two
fixed sides and the maximum area.
\end{quote}

\glava{Introduction}\label{razd.introduction}

What is the triangle with two fixed sides and the maximum square? It
is obvious that in Euclidian geometry this triangle is right-angled.

The aim of this paper is to describe the respective triangle (we
shall call it \emph{the maximum square triangle}) \emph{in
Lobachevskii geometry}. It appears that the maximum square triangle
has a lot of properties corresponding to the properties of the
Euclidian right-angled triangle (see Table~\ref{table.svoystva};
here $\alpha$, $\beta$ and $\gamma$ are the angles which lie
opposite the sides $BC=a$, $AC=b$ and $AB=c$ correspondingly).

\vspace{30pt}

\begin{table}[h]
\caption{Properties of maximum square
triangles}\label{table.svoystva}
\end{table}

\begin{multicols}{2}
\begin{center}
\textsc{Euclidian Geometry} \par
\end{center}
\begin{enumerate}
  \item[1)] $\alpha=\beta+\gamma=\frac\pi 2$;
  \item[2)] the center of the circumcircle coincides with the
middle of the side $BC$;
  \item[3)] $\frac S 2=\frac b 2\cdot\frac c 2$;
  \item[4)] $\cos\alpha=0=\mathrm{const}$;
  \item[5)] $a^2=b^2+c^2$.
\end{enumerate}

\begin{center}
\textsc{Lobachevskii Geometry} \par
\end{center}
\begin{enumerate}
  \item[1)] $\alpha=\beta+\gamma<\frac\pi 2$;
  \item[2)] the center of the circumcircle coincides with the
middle of the side $BC$;
  \item[3)] $\sin\frac S 2=\th\frac b 2\cdot \th\frac c 2$;
  \item[4)] $\cos\alpha=\th\frac b 2\cdot\th\frac c 2\neq\mathrm{const}$;
  \item[5)] $\sh^2\frac a 2=\sh^2\frac b 2+\sh^2\frac c 2$.
\end{enumerate}
\end{multicols}

\vspace{30pt}

So it is \emph{a maximum square triangle} that should be considered
analogous with a Euclidian right-angled triangle.

\vspace{10pt}

\noindent\textsc{Acknowledgements.} The author is grateful to
P.\,V.\,Bibikov for useful discussions and for attention to this
work.

\clearpage

\glava{Poincare Disc Model}\label{razd.model}

We will consider \emph{Poincare disc model of Lobachevskii geometry}
(see~\cite{Efim, Pras}). In this model \emph{Lobachevskii plane} is
the interior of a unite disc; the boundary of this disc is called
\emph{the absolute}. \emph{Points} are Euclidian points;
\emph{lines} are either circle arcs that are orthogonal to the
absolute, or diameters of the absolute (Fig.~\ref{ris.picture1}).
\emph{The angle measure} in Poincare disc model is the same as in
Euclidian geometry.

\begin{figure}[h]
\begin{center}%
\includegraphics[scale=0.8]{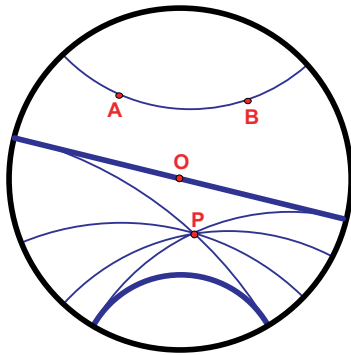}%
\caption{Poincare disk model}\label{ris.picture1}%
\end{center}%
\end{figure}

\emph{A triangle} consists of the circle arcs and the sum of its
angles is less than $\pi$. Let $\delta$ be \emph{the defect of a
triangle}, i.e., $\delta=\pi-\alpha-\beta-\gamma$, where $\alpha$,
$\beta$ and $\gamma$ are the angles of a triangle. It is clear that
the defect of a triangle satisfies the following:

1)  $\delta>0$;

2)  if the triangles $\triangle_1$ and $\triangle_2$ are equal, then
$\delta_1=\delta_2$;

3)  if the triangle $\triangle$ is decomposed into the triangles
$\triangle_1$ and $\triangle_2$, then $\delta=\delta_1+\delta_2$.

So the defect of a triangle satisfies the axioms of \emph{the
square}. It is proved (see~\cite{Nord}) that in Lobachevskii
geometry
$$
S(\triangle)=\delta=\pi-\text{\textbf{the sum of the angles}}.$$

One can see that there is a huge difference between squares in
Euclidian and Lobachevskii geometries. That is why a lot of
Euclidian problems connected with the square become more difficult
and interesting in Lobachevskii geometry.

\glava{Equidistant of the Equal Squares}\label{razd.equid}

\begin{defen}
Let $p$ be a non-Euclidian line. \emph{An equidistant with the base
$p$} is a set of points that are in a fixed half-plane and at a
fixed distance from $p$.
\end{defen}

In other words an equidistant in Lobachevskii geometry is the analog
of a Euclidian parallel line.

In Poincare disc model equidistants are either circle arcs, or
chords of the absolute (see~\cite{Efim, Pras}).

\begin{figure}[h]
\begin{center}
\includegraphics[scale=0.8]{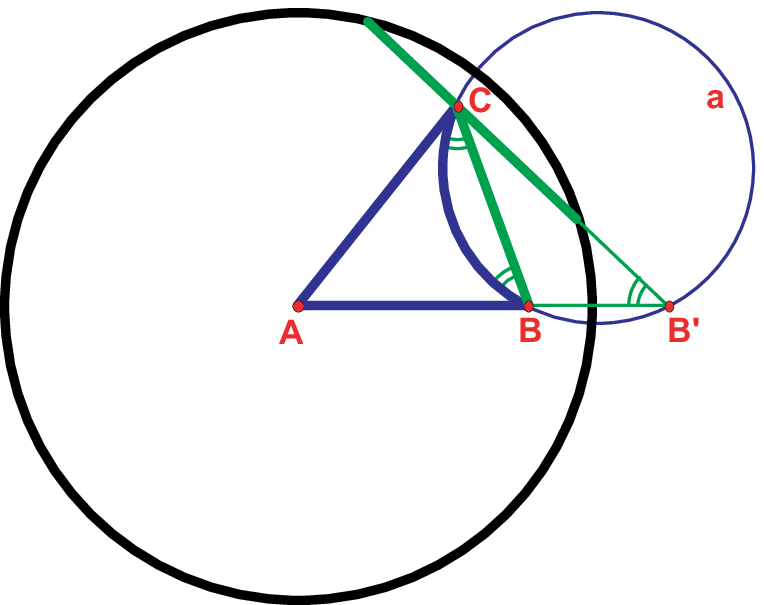}
\caption{The Shvartsman theorem}\label{ris.picture2}%
\end{center}
\end{figure}

\begin{theorem}[O.\,V.\,Shvartsman,~\cite{Shvar}]
Suppose $AB$ is a non-Euclidian segment and $A$ coincides with the
center of the absolute. Let $B'$ be the point symmetric to $B$ with
respect to the absolute\footnote{I.e., $B'$ is the image of $B$
under the inversion with respect to the absolute.} and let $\lambda$
be the chord of the absolute, the continuation of which passes
through $B'$; then for any point $C\in\lambda$ we have
$S(ABC)=2\tau=\mathrm{const}$, where $\tau=\angle AB'C$.
\end{theorem}

\begin{proof}
Let $BC$ be the Euclidian segment and let $a$ be the circle that
passes through $B$, $B'$ and $C$. Then $a$ is orthogonal to the
absolute (see~\cite{Zasl}) so the intersection of $a$ with the
Poincare disc defines the non-Euclidian line
(Fig.~\ref{ris.picture2}).

It is obvious that the angle between the segment $BC$ and the arc
$\smile BC$ of circle $a$ is equal to $\tau$. Therefore, the sum of
the \emph{Euclidian} angles in \emph{Euclidian} triangle $ABC$
equals $\alpha+\beta+\gamma+2\tau=\pi$. Finally, we get
$$S(ABC)=\pi-(\alpha+\beta+\gamma)=2\tau=\mathrm{const}.$$
\end{proof}

\begin{defen}
Let us remember that a chord of the absolute is an equidistant in
Lobachevskii geometry. The chord that satisfies the conditions of
the Shvartsman theorem will be called \emph{the equidistant of the
equal squares for segment $AB$}.
\end{defen}

One can see that the Shvartsman theorem is the analog of a
well-known fact of Euclidian geometry: \emph{the set of points $C$
such that $S(ABC)=\mathrm{const}$ is a line parallel to $AB$}.

\glava{Maximum square triangles and their
properties}\label{razd.trian}

Now we can solve the main problem: \emph{find a non-Euclidian
triangle $ABC$ with two fixed sides $AB$ and $AC$, and the maximum
square}.

Without loss of generality it can be assumed that vertex $A$
coincides with the center of the absolute. Let us fix the side $AB$;
then the vertex $C$ lies on circle $\omega$ with center $A$ and
fixed radius. Consider point $B'$ symmetrical to $B$ with respect to
the absolute and equidistant of the equal squares $\lambda$ for
fixed segment $AB$. By $S/2$ denote the angle between $AB'$ and
$\lambda$. By the Shvartsman theorem triangle $ABC$ with sides $AB$
and $AC$ such that $S(ABC)=S$ exists if and only if the chord
$\lambda$ intersects the circle $\omega$ (Fig.~\ref{ris.picture3}).
So we should find the equidistant $\lambda_{\max}$ that intersects
$\omega$ and forms the maximum angle with $AB'$. It is obvious that
$\lambda_{\max}$ is \emph{a tangent line to $\omega$}. So to build
the maximum square triangle $ABC$ we must build the tangent line
$B'C$ to circle $\omega$.

\begin{figure}[h]
\begin{center}
\includegraphics[scale=0.8]{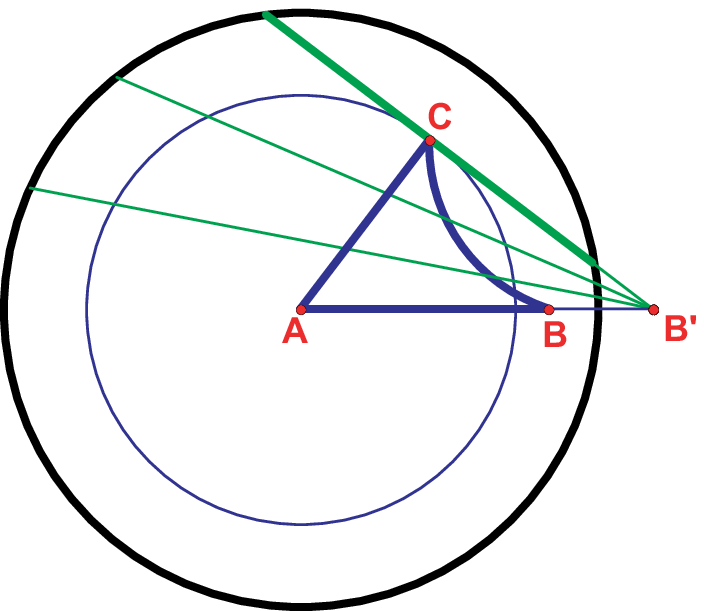}
\caption{Maximum square triangle}\label{ris.picture3}%
\end{center}%
\end{figure}

Since a maximum square triangle is an analog of a Euclidian
right-angled triangle, it will be interesting to describe this
analogy using corresponding properties of these triangles (see
Table~\ref{table.svoystva}). We shall state these properties in the
following theorem.

\begin{theorem}\label{th.svoystva}
Suppose $ABC$ is triangle with fixed sides $AC=b$ and $AB=c$; then
the following conditions are equivalent:

\textup{(}0\textup{)} $ABC$ has the maximum square;

\textup{(}1\textup{)} $\alpha=\beta+\gamma<\frac\pi 2$;

\textup{(}2\textup{)} the center of the circumcircle coincides with
the middle of the side $BC$;

\textup{(}3\textup{)} $\sin\frac S 2=\th\frac b 2\cdot \th\frac c
2$;

\textup{(}4\textup{)} $\cos\alpha=\th\frac b 2\cdot\th\frac c
2\neq\mathrm{const}$;

\textup{(}5\textup{)} $\sh^2\frac a 2=\sh^2\frac b 2+\sh^2\frac c
2$.
\end{theorem}

\begin{proof}
We will use the construction described above.

$(0)\Leftrightarrow(1)$ First suppose that triangle $ABC$ has the
maximum square. Then $\angle ACB'=\frac\pi 2$ and we get
$\tau+\alpha=\frac\pi 2$ $\Leftrightarrow$
$(\pi-\alpha-\beta-\gamma)+2\alpha=\pi$, so $\alpha=\beta+\gamma$.
It is obvious that $\alpha<\frac\pi 2$.

Secondly suppose that the condition $\alpha=\beta+\gamma$ holds.
Then the \emph{Euclidian} angle $\angle ACB'$ in the triangle $AB'C$
is equal to $\pi-\alpha-\frac S 2=\frac\pi 2$. So the line $B'C$ is
tangent to $\omega$ (Fig.~\ref{ris.picture3}) and the triangle $ABC$
has the maximum square.

$(1)\Leftrightarrow(2)$ This statement is well known and belongs to
the absolute geometry.

$(0)\Leftrightarrow(3)$ If the triangle $ABC$ has the maximum
square, then $\sin\frac S 2=\frac{AC_E}{AB'_E}$, where by $AC_E$ and
$AB'_E$ we denote the \emph{Euclidian} lengths of the
\emph{Euclidian} segments $AC$ and $AB'$ respectively. As $B'$ is
symmetric to $B$ with respect to the absolute, then
$\frac{1}{AB'_E}=AB_E$. It is proved in~\cite{Pras} that
$AB_E=\th\frac c 2$ and $AC_E=\th\frac b 2$. Finally, we get
$\sin\frac S 2=AC_E\cdot AB_E=\th\frac b 2\cdot \th\frac c 2$.

The converse statement follows from \emph{the sinus theorem}:
$\frac{AB'_E}{\sin\angle ACB'}=\frac{AC_E}{\sin S/2}$. If we combine
this with (3), we get $\sin\angle ACB'=1$. So $\angle ACB'=\frac\pi
2$ and the line $B'C$ is tangent to $\omega$.

$(0)\Leftrightarrow(4)$ This equivalence is proved in the same way
as the previous one.

$(4)\Leftrightarrow(5)$ To prove this statement, we use \emph{the
cosine theorem} (see~\cite{Pras}): $\ch a=\ch b\ch c-\sh b\sh c
\cos\alpha$. If we replace $\cos\alpha$ by $\th\frac b
2\cdot\th\frac c 2$ in the cosine theorem, we obtain (5). The
converse statement is proved in the same way.
\end{proof}

Thus most of the right-angled triangle properties correspond to
those of the maximum square triangle (and not the right-angled one,
as might have been expected) in Lobachevskii geometry. So it is
\emph{the maximum square triangle} that should be considered
analogous with the Euclidian right-angled triangle.

\begin{rem}
1. Using the proof of the (0) and (3) properties equivalence, we get
the formula to find out the square of an \emph{arbitrary}
non-Euclidian triangle through its two sides and the angle between
them, i.e., the analog of the Euclidian formula
$S=\frac{bc\sin\alpha}{2}$: $$\ctg\frac S 2=\frac{\cth\frac b
2\cdot\cth\frac c 2-\cos\alpha}{\sin\alpha}.$$

Besides, the Shvartsman theorem makes it easy to prove many other
non-Euclidian formulae connected with the triangle square. It also
explains why it is the half of the square and the halves of the
sides that these formulae contain.

2. Properties (1)--(5) from Theorem~\ref{th.svoystva} transform into
corresponding Euclidian properties (see Table~\ref{table.svoystva})
as $b,c\to 0$. This again shows that a maximum square triangle is an
analog of a Euclidian right-angled triangle in Lobachevskii
geometry.

3. However there are some differences between these two types of
triangles. The most extraordinary one is property (4):
$\cos\alpha=\th\frac b 2\cdot\th\frac c 2$. First of all, the angle
$\alpha$ depends on the sides $b$ and $c$, whereas in a Euclidian
right-angled triangle $\alpha=\frac\pi 2=\mathrm{const}$. Second, as
$b,c\to\infty$, the angle $\alpha$ tends to 0! Moreover, the longer
sides $b$ and $c$ are, the smaller the angle $\alpha$ is
(Fig.~\ref{ris.picture4}).

\begin{figure}[h]
\begin{center}
\includegraphics[scale=0.8]{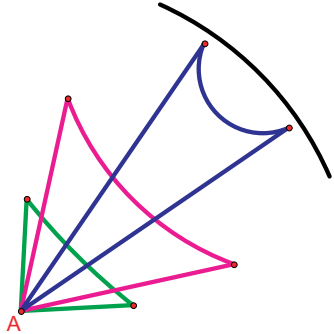}
\caption{Maximum square triangles}\label{ris.picture4}%
\end{center}%
\end{figure}

4. It is natural to call property (5) \emph{the non-Euclidian
pythagorean theorem}, because this relation looks more like the
Euclidian pythagorean theorem $a^2=b^2+c^2$ than the relation $\ch
a=\ch b\ch c$ in the non-Euclidian right-angled triangle.
\end{rem}

\glava{Application: the isoperimetric problem}\label{razd.izoper}

One of the most famous and interesting extremum problems is the
so-called isoperimetric problem: \emph{find a fixed length curve
that bounds a figure of the maximum square}. The answer in Euclidian
geometry is well known: this curve is \emph{the circle}
(see~\cite{Prot}). Using the properties of the maximum square
triangle, we shall now solve this problem in Lobachevskii geometry.

\begin{theorem}
Any fixed length curve that bounds a figure of a maximum square is a
circle.
\end{theorem}

\begin{proof}
Let $F$ be a figure of maximum square $S$ bounded by the curve $f$
of length $l$ (the existence of $F$ is proved in the same way as in
Euclidian geometry; see~\cite{Prot}). It is obvious that $F$ is
convex. Let $BC$ be \emph{the diameter} of $F$, i.e., $BC$ halves
the perimeter of $F$. Then $BC$ also halves the square of $F$.

Let us show that all points of $f$ are at a fixed distance from the
middle $O$ of segment $BC$. Let $A\in f$ be an arbitrary point. We
claim that triangle $ABC$ has the maximum square. Indeed, assume the
converse. Consider the half of $F$ bounded by $BC$ and $f$, and
containing $A$. Let us fix sides $AB$ and $AC$ and change the angle
$\angle BAC$ so that the square of triangle $ABC$ be maximum. Then
we get the figure with the perimeter $l/2$ and the square larger
than $S/2$. If we reflect this figure with respect to line $BC$, we
get figure $F'$ with the perimeter $l$ and the square larger than
$S$. This contradiction proves that triangle $ABC$ has the maximum
square. Using property (2) of Theorem~\ref{th.svoystva}, we get
$OA=OB=OC$. So, the curve $f$ is a circle.
\end{proof}

\textsc{Lyceum "Second School", Moscow, Russia}

\emph{E-mail address}: \texttt{matmail@bk.ru}

\end{document}